\documentclass{amsart}

\usepackage{amsmath, amssymb, amsthm}

\newtheorem{theorem}{Theorem}
\newtheorem{lemma}[theorem]{Lemma}
\newtheorem{proposition}[theorem]{Proposition}

\newtheorem{claim}[theorem]{Claim}
\newtheorem{subclaim}[theorem]{Subclaim}

\theoremstyle{definition}
\newtheorem{definition}[theorem]{Definition}
\newtheorem{remark}[theorem]{Remark}

\newcommand{\bb}{\mathbb}

\newcommand{\tp}{\mathrm{tp}}

\title{Disjoint type graphs with no short odd cycles}
\author{Chris Lambie-Hanson}
\address{Department of Mathematics and Applied Mathematics \\
Virginia Commonwealth University \\
Richmond, VA 23284 \\ United States}
\email{cblambiehanso@vcu.edu}

\begin{document}
\date{\today}
\begin{abstract}
	In this note, we provide a proof of a technical result of
	Erd\H{o}s and Hajnal about the existence of disjoint type graphs with no
	odd cycles. We also prove that this result is sharp in a certain sense.
\end{abstract}
\maketitle

The purpose of this note is to provide a proof of a result of Erd\"{o}s and
Hajnal about the existence of disjoint type graphs with no short odd cycles.
As far as we know, a proof of this result has never been published, though
forms of it are stated in a number of publications (cf.\
\cite[Theorem 7.4]{erdos_hajnal_chromatic_number} and
\cite[Lemma 1.1(d)]{erdos_hajnal_szemeredi}). If $\kappa$ is an uncountable cardinal,
then graphs of this form provide, again as far as we know,
the only known $\mathsf{ZFC}$ examples of graphs with size and chromatic number
$\kappa$ and arbitrarily high odd girth.

Before we state and prove the main result, we need some definitions and
conventions. First, if $n$ is a positive integer, we will sometimes
think of elements of $[\mathrm{Ord}]^n$ as strictly increasing sequences of
length $n$. So, for instance, if $a \in [\mathrm{Ord}]^n$ and $i < n$,
then $a(i)$ is the unique element $\alpha \in a$ such that $|a \cap \alpha| = i$.
All graphs considered here will be simple undirected graphs. If $G$ is a graph,
then $V(G)$ denotes its vertex set and $E(G)$ denotes its edge set.

\begin{definition}
	Let $n$ be a positive integer. A \emph{disjoint type of width $n$} is a
	function $t : 2n \rightarrow 2$ such that
	\[
		|t^{-1}(0)| = |t^{-1}(1)| = n.
	\]
	If $a,b \in [\mathrm{Ord}]^n$ are disjoint and $a \cup b$ is enumerated in
	increasing order as $\{\alpha_i \mid i < 2n\}$, then we say that the
	\emph{type of $a$ and $b$} is $t$, denoted $\tp(a,b) = t$, if
	\[
		a = \{\alpha_i \mid i \in t^{-1}(0)\}
	\]
	and
	\[
		b = \{\alpha_i \mid i \in t^{-1}(1)\}.
	\]
	Let $\hat{t}$ denote the disjoint type of width $n$ denoted by letting
	$\hat{t}(i) = 1 - t(i)$ for all $i < 2n$. It is evident that, if
	$a,b \in [\mathrm{Ord}]^n$ are disjoint and $\tp(a,b) = t$, then
	$\tp(b,a) = \hat{t}$.
\end{definition}

A type $t$ of width $n$ will sometimes be represented by a binary string
of length $2n$ in the obvious way.
We will particularly be interested in the following family of types.

\begin{definition}
	Let $1 \leq s < n < \omega$. Then $t^n_s$ is the disjoint type of width $n$
	whose binary sequence representation consists of $s$ copies of `0', followed
	by $n - s$ copies of `01', followed by $s$ copies of `1'. More formally,
	$t^n_s$ is defined by letting, for all
	$i < 2n$,
	\[
		t^n_s(i) =
		\begin{cases}
			0 & \text{if } i < s \\
			0 & \text{if } s \leq i < 2n - s \text{ and } i - s \text{ is even} \\
			1 & \text{if } s \leq i < 2n - s \text{ and } i - s \text{ is odd} \\
			1 & \text{if } i \geq 2n - s.
		\end{cases}
	\]
	For example, $t^5_2 = 0001010111$.
\end{definition}

\begin{definition}
	Suppose that $n$ is a positive integer, $\beta$ is an ordinal, and
	$t$ is a disjoint type of width $n$. The graph $G(\beta, t)$ is defined as
	follows. Its vertex set is $V(G(\beta, t)) = [\beta]^n$. Given
	$a,b \in [\beta]^n$, we put the edge $\{a,b\}$ into $E(G(\beta, t))$
	if and only if $a$ and $b$ are disjoint and $\tp(a,b) \in \{t, \hat{t}\}$.
\end{definition}

Before we get to our main result, we need a basic lemma. Given a function $f$
from a natural number to $\bb{Z}$, let $\max(f)$ and $\min(f)$ denote the
maximum and minimum values attained by $f$, respectively.

\begin{lemma} \label{discrepancy_lemma}
	Suppose that $k$ is a positive integer and $f:k \rightarrow \bb{Z}$ is a
	function such that
	\begin{itemize}
		\item $f(0) = 0$ and
		\item $|f(i+1) - f(i)| = 1$ for all $i < k$.
	\end{itemize}
	Then $\max(f) - \min(f) < k$.
\end{lemma}

\begin{proof}
	The proof is by induction on $k$. If $k = 1$, then $\max(f) = \min(f) = f(0) = 0$.
	Suppose that $k > 0$ and we have proven the lemma for $k - 1$. Fix
	$f:k \rightarrow \bb{Z}$, and let $f^- = f \restriction (k-1)$.
	If $f(k) - f(k-1) = 1$, then we have $\max(f) \leq max(f^-) + 1$ and
	$\min(f) = \min(f^-)$, so, applying the induction hypothesis to $f^-$,
	we obtain
	\[
		\max(f) - \min(f) \leq 1 + \big(\max(f^-) - \min(f^-)\big) < 1 + (k-1) = k.
	\]
	If $f(k) - f(k-1) = -1$, then we have $\max(f) = \max(f^-)$ and
	$\min(f) \geq \min(f^-) - 1$, so
	\[
		\max(f) - \min(f) \leq \big(\max(f) - \min(f^-)\big) + 1 < (k-1) + 1 = k.
	\]
\end{proof}

We are now ready for the main result of this note. The proof is rather technical;
we recommend that the reader first draw some pictures to convince themselves of
the truth of the theorem in the special case $s = 1$, $n = 3$ (this pair does not satisfy
$n > 2s^2 + 3s +1$, but the conclusion of the theorem still holds). This will
help the reader to get a feel for the problem and motivate the calculations
in the proof. We also note that the lower bound of $2s^2 + 3s + 1$ is probably
not optimal and can likely be improved with a more careful analysis.
Since a precise lower bound for $n$ is not necessary for our desired applications
(cf.\ \cite{clh_finite_subgraphs}), the primary interest of the result for us
is the fact that such a lower bound exists at all.

\begin{theorem} \label{main_thm}
	Suppose that $s$ and $n$ are positive integers with $n > 2s^2 + 3s + 1$, and
	suppose that $\beta$ is an ordinal. Then the graph $G(\beta, t^n_s)$ has no
	odd cycles of length $2s + 1$ or shorter.
\end{theorem}

\begin{proof}
	Let $t = t^n_s$, $V = [\beta]^n$, $E = E(G(\beta, t))$, and $G = G(\beta, t) =
	(V, E)$. We begin by making some preliminary observations. If
	$\{a,b\} \in E$, then either $\tp(a, b) = t$ or $\tp(a, b) = \hat{t}$.
	If $\tp(a, b) = t$, then, for all $i$ with $s < i < n$, we have
	\[
		b(i - s - 1) < a(i) < b(i - s).
	\]
	If $\tp(a,b) = \hat{t}$, then, for all $i < n - s - 1$, we have
	\[
		b(i + s) < a(i) < b(i + s + 1).
	\]
	Suppose that $k$ is a positive integer and $P = \langle a_0, \ldots, a_k \rangle$
	is a path of length $k$ in $G$. For $j \leq k$, let
	\[
		U_j(P) = \big\{i < j \mid \tp(a_i, a_{i+1}) = t\big\},
	\]
	and
	\[
		D_j(P) = \big\{i < j \mid \tp(a_i, a_{i+1}) = \hat{t}\big\}.
	\]
	Intuitively, $U_j(P)$ is the set of steps ``up" in the path among the first
	$j$ steps, and $D(P)$ is the set of steps ``down" among the first $j$ steps.
	Then set $u_j(P) = |U_j(P)|$ and $d_j(p) = |D_j(P)|$; note
	that $u_j(P) + d_j(P) = j$ for all $j \leq k$.

	\begin{claim} \label{bounding_claim}
		Suppose that $1 \leq k \leq 2s +1$ and $P = \langle a_0, \ldots, a_k \rangle$
		is a path in $G$. Let $u = u_k(P)$ and $d = d_k(P)$. Then there is
		$i < n$ such that
		\[
			a_k\big(i - u(s+1) + ds\big) < a_0\big(i\big) < a_k\big(i - us + d(s+1)\big).
		\]
	\end{claim}

	\begin{remark}
		Implicit in the statement of the claim is the assertion that
		\[
			0 \leq i - u(s+1) + ds < i - us + d(s+1) < n,
		\]
		the truth of which will follow readily from the proof.
	\end{remark}

	\begin{proof}[Proof of Claim \ref{bounding_claim}.]
		Define a function $f:k+1 \rightarrow \bb{Z}$ by letting
		$f(j) = u_j(P) - d_j(P)$ for every $j \leq k$. Then $f$ satisfies the
		hypotheses of Lemma \ref{discrepancy_lemma}, so, letting $M = \max(f)$ and
		$m = \min(f)$, we have $M - m \leq k \leq 2s + 1$.

		Let $i = M(s+1)$. Note that
		\[
			M(s+1) \leq (2s + 1)(s + 1) = 2s^2 + 3s + 1,
		\]
		so we certainly have $i < n$.

		\begin{subclaim} \label{inductive_subclaim}
			For every $0 < j \leq k$, we have
			\[
				a_j\big(i - sf(j) - u_j(P)\big) < a_0\big(i\big) < a_j\big(i - sf(j) +
				d_j(P)\big).
			\]
		\end{subclaim}

		\begin{remark}
			Implicit in the statement of this subclaim is the assertion that, for
			each $0 < j \leq k$, we have
			\[
				0 \leq i - sf(j) - u_j(P) < i - sf(j) + d_j(P) < n.
			\]
			This will follow readily from the proof.
		\end{remark}

		\begin{proof}[Proof of Subclaim \ref{inductive_subclaim}.]
			We proceed by induction on $j$. We begin by proving the subclaim for $j = 1$.
			Suppose first that $\tp(a_0, a_1) = t$, so $f(1) = 1$, $u_1(P) = 1$, and
			$d_1(P) = 0$. Then $M \geq 1$,
			so $i \geq s + 1$. Therefore, since $\tp(a_0, a_1) = t$, the preliminary
			observations at the beginning of the proof of the theorem imply that
			\[
				a_1(i - s - 1) < a_0(i) < a_1(i -s),
			\]
			as desired.

			If, on the other hand, $\tp(a_0, a_1) = \hat{t}$, and hence $f(1) = -1$, $u_1(P) =
			0$, and $d_1(P) = 1$, then $m \leq - 1$. Therefore, we have
			$M \leq 2s$, so $i = M(s+1) \leq 2s^2 + 2s < n - s - 1$. Therefore, since
			$\tp(a_0, a_1) = \hat{t}$, the preliminary observations at the beginning
			of the proof imply that
			\[
				a_1(i + s) < a_0(i) < a_1(i + s + 1),
			\]
			as desired.

			Now suppose that $0 < j < k$ and we have established that
			\[
				a_j\big(i - sf(j) - u_j(P)\big) < a_0\big(i\big) < a_j\big(i - sf(j) +
				d_j(P)\big).
			\]
			We will prove the corresponding statement for $j + 1$. Suppose to begin
			that $\tp(a_j, a_{j+1}) = t$, so $f(j+1) = f(j)+1$, $u_{j+1}(P) =
			u_j(P) + 1$, and $d_{j+1}(P) = d_j(P)$. In this case, it follows that
			$f(j) \leq (M-1)$ and $u_j(P) \leq (M-1)$. In particular, we have
			\[
				i - sf(j) -u_j(P) \geq M(s+1) - s(M-1) - (M-1) = s + 1 > s.
			\]
			Therefore, by the preliminary observations, we have
			\begin{align*}
				a_{j+1}\big(i - sf(j) - u_j(P) - s - 1\big) &< a_j\big(i - sf(j) - u_j(P)\big) \\
				a_{j+1}\big(i - sf(j+1) - u_{j+1}(P)\big) &< a_j\big(i - sf(j) - u_j(P)\big)
			\end{align*}
			and
			\begin{align*}
				a_j\big(i - sf(j) + d_j(P)\big) &< a_{j+1}\big(i - sf(j) + d_j(P) - s\big) \\
				a_j\big(i - sf(j) + d_j(P)\big) &< a_{j+1}\big(i - sf(j+1) + d_{j+1}(P)\big).
			\end{align*}
			Combining these inequalities with the inductive hypothesis yields
			\[
				a_{j+1}\big(i - sf(j+1) - u_{j+1}(P)\big) < a_0(i)
				< a_{j+1}\big(i - sf(j+1) + d_{j+1}(P)\big),
			\]
			as desired.

			On the other hand, suppose that $\tp(a_j, a_{j+1}) = \hat{t}$, so
			$f(j+1) = f(j) - 1$, $u_{j+1}(P) = u_j(P)$, and $d_{j+1}(P) = d_j(P) + 1$.
			In this case, it follows that $f(j) \geq (m + 1)$ and $d_j(P) \leq -(m+1)$.
			In particular, we have
			\[
				i - sf(j) + d_j(P) \leq i - s(m + 1) - (m+1) = i - (m+1)(s+1).
			\]
			We know that $M-m \leq 2s + 1$, so $m + 1 \geq M - 2s$. As a result, the
			above inequality becomes
			\[
				i - sf(j) + d_j(P) \leq M(s+1) - (M-2s)(s+1) = 2s^2 + 2s < n - s - 1.
			\]
			Therefore, by the preliminary observations, we have
			\begin{align*}
				a_{j+1}\big(i - sf(j) - u_j(P) + s\big) &< a_j\big(i - sf(j) - u_j(P)\big) \\
				a_{j+1}\big(i - sf(j+1) - u_{j+1}(P)\big) &< a_j\big(i - sf(j) - u_j(P)\big)
			\end{align*}
			and
			\begin{align*}
				a_j\big(i - sf(j) + d_j(P)\big) < a_{j+1}\big(i - sf(j) + d_j(P) + s + 1\big) \\
				a_j\big(i - sf(j) + d_j(P)\big) < a_{j+1}\big(i - sf(j+1) + d_{j+1}(P)\big).
			\end{align*}
			Combining these inequalities with the inductive hypothesis yields
			\[
				a_{j+1}\big(i - sf(j+1) - u_{j+1}(P)\big) < a_0\big(i\big) <
				a_{j+1}\big(i - sf(j+1) + d_{j+1}(P)\big),
			\]
			as desired, finishing the proof of the subclaim.
		\end{proof}
		Since $f(k) = u_k(P) - d_k(P)$, we have
		\[
			i - u(s+1) + ds = i - sf(k) - u_k(P) \text{\ \ \ and \ \ \ }
			i - us + d(s+1) = i - sf(k) + d_k(P).
		\]
		Therefore, the claim follows immediately from Subclaim \ref{inductive_subclaim}.
	\end{proof}

	Now suppose for sake of contradiction that $G$ has an odd cycle of length
	$2s + 1$ or shorter. In other words, there is a positive integer $k \leq s$
	and a path $C = \langle a_0, \ldots, a_{2k + 1} \rangle$ with
	$a_0 = a_{2k + 1}$. Let $u = u_k(C)$ and $d = d_k(C)$. Note that
	$u + d = 2k + 1$. Apply Claim \ref{bounding_claim} to find $i < n$
	such that
	\[
		a_{2k + 1}\big(i - u(s + 1) + ds\big) < a_0(i) < a_{2k + 1}\big(i - us +
		d(s + 1)\big).
	\]
	Since $a_0 = a_{2k + 1}$, this reduces to
	\[
		i - u(s + 1) + ds < i < i - us + d(s+1).
	\]
	Cancelling $i$ from all three terms yields
	\[
		ds - u(s + 1) < 0 < d(s+1) - us.
	\]
	Since $d$ and $u$ are both non-negative integers, this implies that they
	are both nonzero. Therefore,
	the left inequality gives us
	\[
		\frac{d}{u} < \frac{s+1}{s}
	\]
	and the right inequality gives us
	\[
		\frac{s}{s+1} < \frac{d}{u},
	\]
	so we have
	\[
		\frac{s}{s+1} < \frac{d}{u} < \frac{s+1}{s}.
	\]
	In particular, $\frac{d}{u}$ is close to $1$. But we know that $d + u = 2k + 1$;
	the assignments of values to $d$ and $u$ subject to this constraint that put $\frac{d}{u}$
	closest to 1 are either $d = k$ and $u = k + 1$ or vice versa.
	But $k \leq s$, so, if $d = k$ and $u = k + 1$, then
	\[
		\frac{d}{u} \leq \frac{s}{s+1}
	\]
	and, if $d = k + 1$ and $u = k$, then
	\[
		\frac{d}{u} \geq \frac{s+1}{s}.
	\]
	Either possibility gives us a contradiction, so we are done.
\end{proof}

We end this note by making a few further observations about these disjoint
type graphs. We first point out a minor error in the literature. In
\cite[Remark 1]{avart_luczak_rodl}, the authors write, using slightly
different terminology, that, for any positive integer $n \geq 3$, the graph
$G(\beta, t^n_1)$ has no odd cycles of length less than $2\lceil n/2 \rceil$.
This is true for $n = 3$ but false for every larger value of $n$;
$G(\beta, t^n_1)$ always has a cycle of length $5$, as long as
$\beta$ is large enough to allow room for the cycle. In fact,
we have the following general result, showing that Theorem \ref{main_thm}
is sharp in a sense.

\begin{proposition}
	Suppose that $0 < s < n < \omega$ and
	\[
		\beta > (n-1)(2s+3) + (2s+1)(2s+2).
	\]
	Then the graph $G(\beta, t^n_s)$ has a cycle of length $2s + 3$.
\end{proposition}

\begin{proof}
	Let $m = 2s + 3$. We will define a path $\langle a_0, a_1, \ldots,
	a_m \rangle$ in $G(\beta, t^n_s)$ with $a_m = a_0$.
	First define $a_0 = a_m$ by letting $a_m(i) = im$ for all
	$i < n$. The definition of each of the remaining elements of the cycle depends on
	the parity of its index. For $j$ with $0 < j \leq s + 1$, define
	$a_{2j - 1}$ by setting
	\[
		a_{2j -1}(i) = (i + s + j)m - (2j - 1)
	\]
	for all $i < n$, and define $a_{2j}$ by setting
	\[
		a_{2j}(i) = (i + j)m - 2j
	\]
	for all $i < n$. The following facts are easily verified and left to the reader.
	\begin{itemize}
		\item For all $j \leq s$, $\tp(a_{2j}, a_{2j + 1}) = t^n_s$.
		\item For all $j \leq s$, $\tp(a_{2j + 1}, a_{2j + 2}) = \hat{t}^n_s$.
		\item $\tp(a_{2s + 2}, a_m) = \hat{t}^n_s$.
		\item The largest element of any of the vertices in the cycle is
		\[
			a_{2s + 1}(n-1) = (n-1)(2s+3) + (2s+1)(2s+2).
		\]
	\end{itemize}
	Therefore, $\langle a_0, a_1, \ldots, a_m \rangle$ forms a cycle
	of length $2s + 3$ in $G(\beta, t^n_s)$.
\end{proof}

We conclude by noting the following result, which is one of the primary reasons
for interest in disjoint type graphs. The result is due to Erd\H{o}s and Hajnal
\cite{erdos_hajnal_chromatic_number}; the special case $t = t^3_1$ is due to
Erd\H{o}s and Rado \cite{erdos_rado}. A proof of the full result can be found in
\cite[Theorem 2.1]{avart_luczak_rodl}.

\begin{theorem}
	Suppose that $n$ is a positive integer and $t$ is a disjoint type of width
	$n$. For every infinite cardinal $\kappa$, the graph $G(\kappa, t)$ has
	chromatic number $\kappa$.
\end{theorem}

\bibliographystyle{plain}
\bibliography{bib}

\end{document}